\documentclass[12pt]{amsart}
\usepackage{amsmath,amssymb,amsbsy,amsfonts,latexsym,amsopn,amstext,cite,
                                               amsxtra,euscript,amscd,bm,mathabx}
\usepackage{url}
\usepackage[colorlinks,linkcolor=blue,anchorcolor=blue,citecolor=blue,backref=page]{hyperref}
\usepackage{color}
\usepackage{graphics,epsfig}
\usepackage{graphicx}
\usepackage{float} 
\usepackage[english]{babel}
\usepackage{mathtools}
\usepackage{todonotes}
\usepackage{url}
\usepackage[colorlinks,linkcolor=blue,anchorcolor=blue,citecolor=blue,backref=page]{hyperref}
\DeclareMathAlphabet{\mathmybb}{U}{bbold}{m}{n}

\usepackage[norefs,nocites]{refcheck}

\hypersetup{breaklinks=true}

\usepackage[norefs,nocites]{refcheck}
\usepackage[english]{babel}
\begin{document}

\newtheorem{thm}{Theorem}
\newtheorem{lem}[thm]{Lemma}
\newtheorem{claim}[thm]{Claim}
\newtheorem{cor}[thm]{Corollary}
\newtheorem{prop}[thm]{Proposition} 
\newtheorem{definition}[thm]{Definition}
\newtheorem{rem}[thm]{Remark} 
\newtheorem{question}[thm]{Open Question}
\newtheorem{conj}[thm]{Conjecture}
\newtheorem{prob}{Problem}
\newtheorem{Process}[thm]{Process}
\newtheorem{Computation}[thm]{Computation}
\newtheorem{Fact}[thm]{Fact}
\newtheorem{Observation}[thm]{Observation}

\newtheorem{lemma}[thm]{Lemma}

\newcommand{\GL}{\operatorname{GL}}
\newcommand{\SL}{\operatorname{SL}}
\newcommand{\lcm}{\operatorname{lcm}}
\newcommand{\ord}{\operatorname{ord}}
\newcommand{\Op}{\operatorname{Op}}
\newcommand{\Tr}{\operatorname{Tr}}
\newcommand{\Nm}{\operatorname{Nm}}

\numberwithin{equation}{section}
\numberwithin{thm}{section}
\numberwithin{table}{section}

\numberwithin{figure}{section}

\def\sssum{\mathop{\sum\!\sum\!\sum}}
\def\ssum{\mathop{\sum\ldots \sum}}
\def\iint{\mathop{\int\ldots \int}}

\def\wt {\mathrm{wt}}
\def\Tr {\mathrm{Tr}}

\def\SrA{\cS_r\(\cA\)}

\def\vol {{\mathrm{vol\,}}}
\def\squareforqed{\hbox{\rlap{$\sqcap$}$\sqcup$}}
\def\qed{\ifmmode\squareforqed\else{\unskip\nobreak\hfil
\penalty50\hskip1em\null\nobreak\hfil\squareforqed
\parfillskip=0pt\finalhyphendemerits=0\endgraf}\fi}

\def \ss{\mathsf{s}} 

\def \balpha{\bm{\alpha}}
\def \bbeta{\bm{\beta}}
\def \bgamma{\bm{\gamma}}
\def \blambda{\bm{\lambda}}
\def \bchi{\bm{\chi}}
\def \bphi{\bm{\varphi}}
\def \bpsi{\bm{\psi}}
\def \bomega{\bm{\omega}}
\def \btheta{\bm{\vartheta}}

\newcommand{\bfxi}{{\boldsymbol{\xi}}}
\newcommand{\bfrho}{{\boldsymbol{\rho}}}

 \def \xbar{\overline x}
  \def \ybar{\overline y}

\def\cA{{\mathcal A}}
\def\cB{{\mathcal B}}
\def\cC{{\mathcal C}}
\def\cD{{\mathcal D}}
\def\cE{{\mathcal E}}
\def\cF{{\mathcal F}}
\def\cG{{\mathcal G}}
\def\cH{{\mathcal H}}
\def\cI{{\mathcal I}}
\def\cJ{{\mathcal J}}
\def\cK{{\mathcal K}}
\def\cL{{\mathcal L}}
\def\cM{{\mathcal M}}
\def\cN{{\mathcal N}}
\def\cO{{\mathcal O}}
\def\cP{{\mathcal P}}
\def\cQ{{\mathcal Q}}
\def\cR{{\mathcal R}}
\def\cS{{\mathcal S}}
\def\cT{{\mathcal T}}
\def\cU{{\mathcal U}}
\def\cV{{\mathcal V}}
\def\cW{{\mathcal W}}
\def\cX{{\mathcal X}}
\def\cY{{\mathcal Y}}
\def\cZ{{\mathcal Z}}
\def\Ker{{\mathrm{Ker}}}

\def\NmQR{N(m;Q,R)}
\def\VmQR{\cV(m;Q,R)}

\def\Xm{\cX_{p,m}}

\def \A {{\mathbb A}}
\def \B {{\mathbb A}}
\def \C {{\mathbb C}}
\def \F {{\mathbb F}}
\def \G {{\mathbb G}}
\def \L {{\mathbb L}}
\def \K {{\mathbb K}}
\def \PP {{\mathbb P}}
\def \Q {{\mathbb Q}}
\def \R {{\mathbb R}}
\def \Z {{\mathbb Z}}
\def \fS{\mathfrak S}
\def \fB{\mathfrak B}

\def\Fq{\F_q}
\def\Fqr{\F_{q^r}} 
\def\ovFq{\overline{\F_q}}
\def\ovFp{\overline{\F_p}}
\def\GL{\operatorname{GL}}
\def\SL{\operatorname{SL}}
\def\PGL{\operatorname{PGL}}
\def\PSL{\operatorname{PSL}}
\def\li{\operatorname{li}}
\def\sym{\operatorname{sym}}

\def\Mob{M{\"o}bius }

\def\fF{\EuScript{F}}
\def\M{\mathsf {M}}
\def\T{\mathsf {T}}

\def\e{{\mathbf{\,e}}}
\def\ep{{\mathbf{\,e}}_p}
\def\eq{{\mathbf{\,e}}_q}

\def\\{\cr}
\def\({\left(}
\def\){\right)}

\def\<{\left(\!\!\left(}
\def\>{\right)\!\!\right)}
\def\fl#1{\left\lfloor#1\right\rfloor}
\def\rf#1{\left\lceil#1\right\rceil}

\def\Tr{{\mathrm{Tr}}}
\def\Nm{{\mathrm{Nm}}}
\def\Im{{\mathrm{Im}}}

\def \oF {\overline \F}

\newcommand{\pfrac}[2]{{\left(\frac{#1}{#2}\right)}}

\def \Prob{{\mathrm {}}}
\def\e{\mathbf{e}}
\def\ep{{\mathbf{\,e}}_p}
\def\epp{{\mathbf{\,e}}_{p^2}}
\def\em{{\mathbf{\,e}}_m}

\def\Res{\mathrm{Res}}
\def\Orb{\mathrm{Orb}}

\def\vec#1{\mathbf{#1}}
\def \va{\vec{a}}
\def \vb{\vec{b}}
\def \vh{\vec{h}}
\def \vk{\vec{k}}
\def \vs{\vec{s}}
\def \vu{\vec{u}}
\def \vv{\vec{v}}
\def \vz{\vec{z}}
\def\flp#1{{\left\langle#1\right\rangle}_p}
\def\T {\mathsf {T}}

\def\sfG {\mathsf {G}}
\def\sfK {\mathsf {K}}

\def\mand{\qquad\mbox{and}\qquad}

\title[Character sums over   finite fields]
{Character sums over elements of extensions of finite fields with restricted coordinates}

\author[S. Iyer] {Siddharth Iyer}
\address{School of Mathematics and Statistics, University of New South Wales, Sydney, NSW 2052, Australia}
\email{siddharth.iyer@unsw.edu.au}

\author[I. E. Shparlinski] {Igor E. Shparlinski}
\address{School of Mathematics and Statistics, University of New South Wales, Sydney, NSW 2052, Australia}
\email{igor.shparlinski@unsw.edu.au}

\begin{abstract}  We obtain nontrivial bounds for character sums with multiplicative 
and additive characters over finite fields over elements with restricted coordinate expansion.
In particular, we obtain a nontrivial estimate for such a sum over a finite field analogue of the Cantor set. \end{abstract}


\maketitle

\tableofcontents

\section{Introduction} 
Let $\(\vartheta_1, \ldots, \vartheta_r\)$ be a basis of the finite field
$$
\Fqr =\left \{a_1\vartheta_1+ \ldots +a_r \vartheta_r:~a_1, \ldots, a_r \in \Fq\right\}
$$
of $q^r$ elements over the finite field $\Fq$ of $q$ elements. 

Motivated by a series of recent outstanding results on integers with restricted digital expansion 
in a given basis, there has also been very significant progress in studying elements $\omega \in \Fqr$
with various restrictions on their coordinates $(a_1, \ldots, a_r)$ in the expansion 
$$
\omega = a_1\vartheta_1+ \ldots +a_r \vartheta_r\in \Fqr, 
$$ 
we refer to~\cite{MSW} for a brief outline of such results (both settings on integers and finite fields), 
some new results and further references, in particular on bounds of various character sums over such 
field elements. 

Here, given a set $\cA\subseteq \F_q$, we consider the set 
\begin{equation}\label{eq:Set SrA}
\SrA = \left \{a_1\vartheta_1+ \ldots +a_r \vartheta_r:~a_1, \ldots, a_r \in \cA\right\}
\end{equation}
that is the set of $u \in \Fqr$ whose coordinates are restricted to the set $\cA$.

In particular, one of the natural examples is the case of $q=3$ and $\cA =\{0, 2\}$ 
which leads to a Cantor-like set $\SrA\subseteq \F_{3^r}$. 
 
The main goal of this paper is to estimate mixed character sums 
$$
S_{r}(\cA;\chi,\psi;f_1,f_2) = \sum_{\omega \in \SrA}\chi\(f_1(\omega)\)\psi\(f_2(\omega)\), 
$$
with rational functions $f_1(X), f_2(X) \in \Fqr(X)$, of degrees $d_1$ and $d_2$, respectively, and 
where $\chi$ and $\psi$ are a fixed multiplicative and 
 additive character of $\F_{q^r}$, respectively (with the natural conventions that the poles of $f_1(X)$ and $f_2(X)$ are excluded from summation).
 
 We are especially interested in the case when $\cA$ is of cardinality $\#\cA$ relatively small 
 compared to $q$. In particular, we  are interested in obtaining nontrivial bounds in the case of small
 values of the parameter 
$$
 \rho = \frac{\log \#\cA}{\log q}.
$$

It is well known that such bounds can be used  to study, for example, 
 the distribution of primitive elements in the values of polynomials on elements from 
 $\SrA$ or their pseudorandom properties. Since these applications are quite standard, we do not
 present them here.

\section{Notation and conventions}

Throughout the paper,  we fix the size $q$ of the ground field, and thus also its characteristic  $p$ 
while the parameter $r$ is allowed to grow. 

We also fix an additive character $\psi$ and a multiplicative character~$\chi$ of~$\Fqr$
which are not both principal.

As usual, we use $\ovFq$ to denote the algebraic closure of $\F_q$. It is useful to recall that 
 $\ovFq \subseteq  \ovFp$. 

For a finite set $\cS$, we use $\# \cS$ to denote its cardinality.

We denote by $\log_2 x$ the binary logarithm of $x>0$.

We adopt the Vinogradov symbol $\ll$,  that is, for any quantities $A$ and $B$ we
have the following equivalent definitions:
$$A\ll B~\Longleftrightarrow~A=O(B)~\Longleftrightarrow~|A|\le c B$$
for some  constant $c>0$, which throughout the paper is allowed to depend on the 
degrees $d, d_1,d_2$ and the ground field size~$q$ 
(but not on the main parameter $r$) and the integer parameter $s\ge 1$.

 We also adopt the $o$-notation
$$A=o(B)~\Longleftrightarrow ~|A|\le \varepsilon B$$
for any fixed $\varepsilon>0$ and sufficiently large (depending on $d$, $d_1$, $d_2$, $q$, $s$ 
and $\varepsilon$) values of the parameter $r$

For a rational function $g(X) \in \ovFp(X)$ and an element $w \in\ovFp$ we define $\ord_{w} g$ to be the unique integer so that $(X-w)^{\ord_{w}(g)}g$ extends to a rational function which has no zero or pole at $w$.

We also write 
$$\ep(z)=\exp(2\pi iz/p).
$$
Finally, we also recall our convention that the poles of functions in the arguments of multiplicative and 
additive characters are always excluded from summation.

\section{Main results}

We define the following sets of rational functions.

\begin{definition}  
\label{def: set Q and R} For   integers $d\ge 0$ and $n \ge 2$, 
\begin{itemize}
\item let $\cQ_{d,n}$ be the set of rational functions $g(X)\in\F_{q^r}(X)$ 
of degree at most $d$, which are not an $n$-th power of some rational function in $ \ovFp(X)$.

\item  let $\cR_{d}$ be the set of rational functions $f(X)\in\F_{q^r}(X)$ of degree at most $d$, 
which have at least one pole of order that is not a multiple of $p$.
\end{itemize}
\end{definition}  

We note that we allow $d=0$ in Definition~\ref{def: set Q and R}, that is non-zero constant functions, 
in which case $\cQ_{d,n} = \cR_{d} = \emptyset$. 

We are now ready to present our main result. We recall our convention that implied 
constants are allowed to depend on the integer parameters $d_1$, $d_2$, $q$ and $s$.

For an integer $s \ge 1$ we define
\begin{equation}\label{eq:kappa-s}
\kappa_s(\rho) = \frac{s\rho(2\rho-1)+\rho -1}{4s(s\rho +1)}.
\end{equation}

\begin{thm}\label{thm:gen bound}
Let $\chi$ and $\psi$ be a multiplicative and additive character, respectively, and 
let $f_1(X) , f_2(X)\in \F_{q^r}(X)$.  Assume that at least one of the following conditions holds
\begin{itemize}
\item[(i)] $\chi$ is nonprincipal of order $n$ and $f_1(X) \in \cQ_{d,n}$, 
\item[(ii)] $\psi$ is nonprincipal and $f_2(X) \in\cR_{d}$. 
\end{itemize}
Then  for any fixed integer $s\ge 1$, we have 
$$
S_{r}(\cA;\chi,\psi;f_1,f_2)  \ll \(\# \cA\)^r q^{- r \kappa_s(\rho) } . 
$$  
\end{thm} 

Clearly for any $\rho > 1/2$ we have $\kappa_s(\rho)  >0$ for a sufficiently large $s$.

In particular, with 
$$
\rho =\frac{\log 2}{\log 3}
$$ 
taking $s=5$ in Theorem~\ref{thm:gen bound}
we have the following nontrivial bound for a ``Cantor-like'' set in finite fields. 

\begin{cor}
\label{cor:Cantor} Let $q = 3$ and $\cA  = \{0,2\}$. Under the conditions of Theorem~\ref{thm:gen bound} 
we have 
$$
S_{r}(\cA;\chi,\psi;f_1,f_2)  \ll 2^{\gamma r}
$$
where 
$$
\gamma = 1 - \frac{\log 3}{\log 2}\cdot \kappa_{5}\(\frac{\log 2}{\log 3}\)
= 0.99128\ldots\,.
$$
\end{cor}

We remark that both Theorem~\ref{thm:gen bound} and Corollary~\ref{cor:Cantor}
apply to {\it Kloosterman sums\/} 
$$
 \sum_{\omega \in \SrA} \psi\(a\omega + b \omega^{-1} \), \qquad (a,b) \in  \F_{q^r} \times \F_{q^r}^*, 
$$
over elements of $\SrA$. 

 We note that unfortunately  Theorem~\ref{thm:gen bound} does not apply to polynomials $ f_2$ 
 if either  $\chi$ is principal or $f_1 \not \in \cQ_{d,n}$. 
Hence, we introduce another class of functions which actually originates from~\cite{MSW}.  

\begin{definition}  
 Let $\cP_{d}$ be the set of rational functions $f(X)\in\F_{q^r}(X)$ of 
degree $d$ such that 
for any  $\omega\in \F_{q^r}^*$
the function  
$$
f_{\omega}(X)=f(X+\omega)-f(X)
$$
is not of the form
$$
f_{\omega}(X)=\alpha\(g(X)^p-g(X)\)+\beta X 
$$
for some rational function $g(X)\in \ovFq(X)$  and  $\alpha,\beta\in \ovFq$.
\end{definition}    

We refer to~\cite{MSW} for examples of functions from  $\cP_{d}$. 

For a function $f_2 \in \cP_d$, we are only able to obtain  
a version of Theorem~\ref{thm:gen bound} with $s=1$, and hence we save only 
$$
\kappa_1(\rho) = \frac{2\rho^2-1}{4(\rho +1)}.
$$

\begin{thm}\label{thm:poly bound}
Let $\chi$ and $\psi$ be a multiplicative and additive character, respectively, and 
let $f_1(X) , f_2(X)\in \F_{q^r}(X)$.  Assume that  $\psi$ is nonprincipal and $f_2(X) \in\cP_{d}$. 
Then, we have 
$$
S_{r}(\cA;\chi,\psi;f_1,f_2)  \ll  \(\# \cA\)^r  q^{- r \kappa_1(\rho) } . 
$$  
\end{thm} 

Note that $\kappa_1(\rho) >0$ only for 
$$
\rho > 2^{-1/2} = 0.707106\ldots > \frac{\log 2}{\log 3} =0.63092\ldots , 
$$ 
and hence unfortunately  Theorem~\ref{thm:poly bound} does not apply to the setting of Corollary~\ref{cor:Cantor}.

\section{Ratios and linear combinations of shifts of rational functions}

Various versions of the following results have been well-known, see, for example, 
the proof of~\cite[Theorem~1] {BaCoSh} or of~\cite[Theorem~1]{OstShp}.

It is convenient to introduce the following notation. Given a vector 
$\vv= \(v_{1},\ldots ,v_{2s}\)\in \ovFp^{2s}$ and a rational function $f \in \ovFp(X)$, we
set 
 \begin{equation}\label{eq:Prod f}
P_{\vv,f}(X)= \prod_{i=1}^{s}\frac{f\(X+v_{i}\)}{f\(X+v_{s+i}\)}.
\end{equation}

The implied constants in this section may depend only  on $d = \deg f$  and $s$, 
but are uniform with respect to other parameters, including $q$, and most 
importantly $n$, $r$ and $V$.

\begin{lem}
\label{lem:rat prod}
 Let   $f(X) \in \cQ_{d,n}$  for some integers $d \ge 1$ and $n \geq 2$. 
 For any set $\cV \subseteq \ovFp$ of cardinality $V$, for each integer 
 $s \geq 1$ we have
$$
\# \left\{\vv = \(v_{1},\ldots ,v_{2s}\)\in \cV^{2s}:~P_{\vv,f}(X) 
\not \in \cQ_{2ds,n}\right\} \ll  V^s. 
$$  
\end{lem}

\begin{proof} Without loss of generality, we can assume that all zeros and poles of $f$ 
are of order less than $n$, that is, 
$$
f(X) = \prod_{j=1}^{h}(X-\alpha_{j})^{u_{j}},
$$
where $\alpha_{j} \in \overline{\F_{p}}$ are pairwise distinct and $u_{j} \in \{\pm 1, \ldots, \pm (n-1)\}$, 
$j =1, \ldots, h$.

If $v_{1},\ldots ,v_{2s} \in \cV$ are chosen so that there exist some integers $k$ and $\ell$ with 
 $1 \leq k  \leq 2s$ and $1 \le \ell \le h$ so that 
$$
 v_{k}- \alpha_{\ell} \neq v_{i}- \alpha_{j}
$$
for all $(i,j) \neq (k, \ell)$ then for $\beta = \alpha_{\ell}-v_{k}$ we have
$$\ord_{\beta}\prod_{i=1}^{s}f\(X+v_{i}\)/f\(X+v_{s+i}\)
\equiv u_{\ell} \not \equiv 0 \pmod n$$
and thus, the above rational function is not an $n$-th  power.  

Let $E$ be the number of $\vv = \(v_{1},\ldots ,v_{2s} \)\in \cV^{2s}$ for which $P_{\vv,f}(X) \not \in \cQ_{2ds,n}$. 
Then for each choice of $1 \leq i\leq 2s$ there is some index $k\ne i$,  $1 \leq k\leq 2s$,  such that $v_i-v_k$ 
belongs to the difference set of the set $\{\alpha_1, \ldots, \alpha_h\}$ and thus can take at most $h(h-1) +1 \le d^2$ 
values. In particular, the components of $\vv$ can be partitioned into at most $s$ groups such that differences  of elements within each group belong to the above difference set. This immediately implies that $E\ll V^s$ and 
concludes the proof.  
\end{proof}

We use Lemma~\ref{lem:rat prod} to control sums of multiplicative characters. To control sums of 
additive characters we need its appropriate analogue for linear combinations instead of products 
as in~\eqref{eq:Prod f}. Namely, given a vector 
$\vv= \(v_{1},\ldots ,v_{2s}\)\in \ovFp^{2s}$ and a rational function $f \in \ovFp(X)$, we
set 
 \begin{equation}\label{eq:LinComb f}
L_{\vv,f}(X)= \sum_{i=1}^{s} \(f\(X+v_{i}\)- f\(X+v_{s+i}\)\).
\end{equation}

\begin{definition}
We define the set $\cE$ of {\it exceptional\/} rational functions as the set of 
 rational functions $f(X) \in \F_{p^r}(X)$  such that  there exists $\alpha , \beta \in \overline{\F_{p}}$ and $h(X) \in \overline{\F_{p}}(X)$ so that $f(X) = \alpha(h(X)^p-h(X)) + \beta X$.
\end{definition}

 Then we have the following additive analogue of  Lemma~\ref{lem:rat prod}.

\begin{lem}
\label{lem:rat lincomb}
 Let   $f(X) \in \cR_{d}$  for some integers $d \ge 1$. 
 For any set $\cV \subseteq \ovFp$ of cardinality $V$, for each integer 
 $s \geq 1$ we have
$$
\# \left\{\vv = \(v_{1},\ldots ,v_{2s}\)\in \cV^{2s}:~L_{\vv,f}(X) 
 \in \cE\right\} \ll V^s. 
$$  
\end{lem}

\begin{proof} Clearly, all functions from $\cE$ have a pole of   order that is a multiple of $p$.
It is also clear that if $f_{1},\ldots ,f_{n} \in \ovFp(X)$ are such that $f_{1}$ has a pole at $\alpha \in  \ovFp$ of order $u\ge 1$ and $f_{2},\ldots, f_{n}$ have no 
poles at $w$ then $f_{1}+\ldots +f_{n}$ has a pole at $\alpha$ of the same order $u$. 

This implies that if $L_{\vv,f}(X)  \in \cE$ then  for each choice of $1 \leq i\leq 2s$ there is some index $k\ne i$,  $1 \leq k\leq 2s$,  such that 
$$
  v_i-v_k \in \{\alpha- \gamma:~\gamma\ \text{is a pole of} \ f\}. 
  $$
 Indeed,  otherwise, that is, for other choices of $\(v_{1},\ldots ,v_{2s}\)\in \cV^{2s}$,  if $\alpha$ is a pole $f\in \ovFp(X)$ of order $\ord_\alpha f \not \equiv 0 
\pmod p$, then $f(X+v_i)$ has a pole $\beta = \alpha - v_i$ of the same order and which is not a pole 
of any other function involved in $L_{\vv,f}$. Hence 
$$
\ord_\beta L_{\vv,f} = \ord_\alpha f \not \equiv 0 \pmod p, 
$$
and therefore $ L_{\vv,f} \not \in \cE$.  
We see that the number of such choices  of  $\(v_{1},\ldots ,v_{2s}\)\in \cV^{2s}$ is at most 
$$
2^s \binom{2s}{s} d^s V^s \ll V^s
$$
and the result now follows. 
\end{proof}

\section{Character sums over linear subspaces}

We  need is~\cite[Lemma~3.2]{MSW} which follows instantly from 
the Weil bound for mixed character sums with rational functions
due to  Castro and Moreno~\cite{CaMo} (see also more general results 
of Fu and Wan~\cite[Theorem~5.6]{FW}) and the orthogonality of additive characters.

\begin{lemma}\label{lem:weil-subspace}  
Let $\chi$ and $\psi$ be a multiplicative and additive character, respectively, and 
let $g_1(X) , g_2(X)\in \F_{q^r}(X)$ be rational functions of degrees at most $d$. 
Assume that at least one of the following conditions holds
\begin{itemize}
\item[(i)] $\chi$ is nonprincipal of order $e$ and $g_1(X)  \in \cQ_{d,e}$,
\item[(ii)] $\psi$ is nonprincipal and $g_2(X)  \not \in \cE$. 
\end{itemize}
 Then for any affine subspace $\cL \subseteq  \F_{q^r}$ we have
$$
\sum_{\lambda \in \cL }\chi(g_1(\lambda))\psi\(g_2(\lambda)\) \ll  q^{r/2}.
$$
\end{lemma}

The following result is our main technical tool. 
 
\begin{lemma}
\label{lem:double sum}
Let $\chi$ and $\psi$ be a multiplicative and additive character, respectively, and 
let $g_1(X) , g_2(X)\in \F_{q^r}(X)$ be rational functions of degrees at most $d$. 
Assume that at least one of the following conditions holds
\begin{itemize}
\item[(i)] $\chi$ is nonprincipal of order $e$ and $g_1(X)  \in \cQ_{d,e}$, 
\item[(ii)] $\psi$ is nonprincipal and $g_2(X)  \not \in \cE$. 
\end{itemize}
 Then for a linear space $\cL  \subseteq \F_{p^r}$ of dimension $t$ and  arbitrary  set $\cU    \subseteq \cL$ 
 and $\cV   \subseteq \F_{p^r}$ of cardinalities $U$ and $V$, respectively, 
 for each fixed integer $s \ge 1$ we have
$$
 \sum_{u\in \cU}  \left| \sum_{v \in \cV} \chi(g_{1}(\lambda+v))\psi(g_{2}(\lambda+v))\right|
 \ll    U^{1-1/(2s)} \(q^{t/(2s)}  V^{1/2}  +    q^{r/(4s)}   V\) . 
$$
\end{lemma}   

\begin{proof} Let
$$
S = \sum_{\lambda \in \cU}  \left| \sum_{v \in \cV} \chi(g_{1}(u+v))\psi(g_{2}(u+v))\right|.
$$
Applying the H{\"o}lder inequality, we derive 
\begin{align*}
S^{2s} &  \le U^{2s-1}\sum_{\lambda \in \cU}  \left| \sum_{v \in \cV} \chi(g_{1}(u+v))\psi(g_{2}(u+v))\right|^{2s}\\
&  \le U^{2s-1}\sum_{\lambda \in \cL}  \left| \sum_{v \in \cV} \chi(g_{1}(\lambda+v))\psi(g_{2}(\lambda+v))\right|^{2s}\\
& = U^{2s-1} \sum_{\lambda\in \cL} \, \sum_{\vv= \(v_{1},\ldots ,v_{2s}\)\in \cV^{2s}} 
\chi\(P_{\vv,f}(\lambda)\)\psi\(L_{\vv,f}(\lambda)\)\\
& = U^{2s-1} \sum_{\vv= \(v_{1},\ldots ,v_{2s}\)\in \cV^{2s}}\,  \sum_{\lambda \in \cL} 
\chi\(P_{\vv,f}(\lambda)\)\psi\(L_{\vv,f}(\lambda)\),
\end{align*} 
where $P_{\vv,f}(X)$ and $L_{\vv,f}(X)$ are defined by~\eqref{eq:Prod f} and~\eqref{eq:LinComb f}, 
respectively.

We now see that if at least one of the above conditions~(i) or~(ii) holds that by either 
Lemma~\ref{lem:rat prod} or Lemma~\ref{lem:rat lincomb} we can apply 
Lemma~\ref{lem:weil-subspace} to the inner sum over the linear space $\cL$ for all but $O\(V^s\)$
vectors $\vv \in \cV$, for which we estimate the inner sum trivially as $q^t$. 
Hence 
$$
S^{2s} \ll  U^{2s-1} \(q^t V^{s} + q^{r/2} V^{2s}\), 
$$
and the result follows. 
\end{proof}

\section{Proof of Theorems~\ref{thm:gen bound} and~\ref{thm:poly bound}}

We recall the definition~\eqref{eq:Set SrA}
of  the set $\SrA$, and  some real  positive parameter  $\tau \in [0,1] $ and 
set $t = \fl{\tau r}$, to be chosen later, 
we define the sets
\begin{align*}
&\cU= \left \{a_1\vartheta_1+ \ldots +a_t \vartheta_t:~a_1, \ldots, a_t \in \cA\right\},\\
& \cL = \left \{a_1\vartheta_1+ \ldots +a_t \vartheta_t:~a_1, \ldots, a_t \in \F_{q}\right\},\\
&\cV= \left \{a_{t+1}\vartheta_{t+1}+ \ldots +a_r \vartheta_r:~a_{t+1}, \ldots, a_r \in \cA\right\}, 
\end{align*} 
of cardinalities
$$
 U = q^{\rho t} \ll q^{\rho \tau r}, \qquad 
 L = q^{t} \ll q^{\tau r},\qquad V = q^{\rho (r- t)} \ll q^{\rho(1- \tau) r},
$$
respectively. We can now write 
\begin{align*}
S_{r}(\cA;\chi,\psi;f_1,f_2)& = \sum_{\omega \in \SrA}\chi\(f_1(\omega)\)\psi\(f_2(\omega)\)\\
&= \sum_{u \in \cU}\sum_{v \in \cV}\chi\(f_1(u+v)\)\psi\(f_2(u+v)\).
\end{align*}
Thus
\begin{align*}
\left|S_{r}(\cA;\chi,\psi;f_1,f_2)\right| \leq \sum_{u \in \cU}\left|\sum_{v \in \cV}\chi\(f_1(u+v)\)\psi\(f_2(u+v)\)\right|.
\end{align*}

Under the conditions of Theorem~\ref{thm:gen bound}, 
by Lemma~\ref{lem:double sum} and the above cardinality estimates we have
\begin{align*}
S_{r}(\cA;\chi&,\psi;f_1,f_2) \ll U^{1-1/(2s)} \(q^{t/(2s)}  V^{1/2}  +    q^{r/(4s)}   V\)\\
& \ll q^{\rho \tau r(1 - 1/(2s))+\tau r/(2s)+\rho(1-\tau)r/2} + q^{\rho \tau r(1 - 1/(2s))+r/(4s)+\rho(1-\tau)r}\\
& = q^{r(\tau\rho(1 - 1/(2s)) +\tau(1/(2s)- \rho/2) + \rho/2)}+q^{r(\tau\rho(1 - 1/(2s)) +\rho(1-\tau) + 1/(4s))}.
\end{align*}  
Hence we have
 \begin{equation}\label{eq:Rough}
S_{r}(\cA;\chi,\psi;f_1,f_2)  \ll q^{r\Delta_{s,\rho}(\tau)}
\end{equation}
where 
$$
\Delta_{s,\rho}(\tau)=  \(1-  \frac{1}{2s}\)\rho \tau + \max\left\{ \frac{\tau}{2s} + \frac{\rho(1- \tau)}{2}, 
 \frac{1}{4s} +  \rho(1- \tau) \right \}.
$$
To minimise $\Delta_s(\tau)$ we choose 
$$
\tau_0 = \frac{2s \rho+1}{2(s\rho +1)}  
$$
to equalise the terms inside of the above maximum and compute 
\begin{align*}
\Delta_{s,\rho}(\tau_0) & =   \(1-  \frac{1}{2s}\)\rho \tau_0 +  
 \frac{1}{4s} +  \rho(1- \tau_0)\\
& = \rho + \frac{1}{4s} -   \frac{1}{2s}\rho \tau_0 = \rho  - \frac{1}{4s} (2 \rho \tau_0 -1)
= \rho - \kappa_s(\rho) , 
\end{align*} 
where $\kappa_s(\rho) $ is given by~\eqref{eq:kappa-s}, which together with~\eqref{eq:Rough} concludes the proof 
of Theorem~\ref{thm:gen bound}. 

 To prove Theorem~\ref{thm:poly bound}, we note that for $s=1$,  the above argument still applies for $f_2 \in \cP_d$, 
and the result follows.

\section*{Acknowledgements}

The authors are grateful to the referee for the careful reading of the manuscript and many useful comments.

During the preparation of this work, I.S. was supported in part by the Australian Research Council Grants~DP230100530 and  DP230100534.


\begin{thebibliography}{www}


\bibitem{BaCoSh} W. Banks, A. Conflitti and I. E. Shparlinski,
`Character sums over integers with restricted
$g$-ary digits',   {\it Illinois J. Math.\/},
{\bf 46} (2002), 819--836.

\bibitem{CaMo} F. N. Castro and C. J. Moreno, 
`Mixed exponential sums over finite fields', 
{\it Proc. Amer. Math. Soc.\/}, {\bf 128}  (2000),  2529--2537.

\bibitem{FW} L. Fu and D. Wan, 
`A class of incomplete character sums',
{\it Quart. J. Math.\/},  {\bf 65} (2014), 1195--1211. 

\bibitem{MSW} L. M{\'e}rai, I. E. Shparlinski and A. Winterhof, `Character sums over sparse elements of finite fields', 
{\it Preprint\/}, 2022, (available from  \url{https://arxiv.org/abs/2211.08452}).

 \bibitem{OstShp} A. Ostafe  and I. E. Shparlinski, `Multiplicative character sums and products of sparse integers in residue classes', {\it Period. Math. Hungar.\/},
{\bf 64} (2012),  247--255.


\end{thebibliography}
\end{document}